\documentclass[11pt]{article}
\usepackage[colorlinks]{hyperref}
\usepackage{color}
\usepackage{graphicx}
\usepackage{graphics}
\usepackage{makeidx}
\usepackage{showidx}
\usepackage{latexsym}
\usepackage{amssymb}
\usepackage{verbatim}
\usepackage{amsmath}
\usepackage{amsthm}
\usepackage{amsfonts}
\usepackage{subcaption} 
\usepackage[abs]{overpic}
\usepackage{xcolor,varwidth}
\usepackage{geometry}
\newtheorem{theorem}{Theorem}[section]
\newtheorem{proposition}[theorem]{Proposition}
\newtheorem{remark}[theorem]{Remark}
\newtheorem{claim}[theorem]{Claim}

\newtheorem{definition}[theorem]{Definition}
\newtheorem{corollary}[theorem]{Corollary}
\newtheorem{lemma}[theorem]{Lemma}

\newtheorem{corollaryproof}[theorem]{Corollary of the Proof}

\title{On Anosovity, divergence and bi-contact surgery}
\author{Surena Hozoori}
\newcommand{\Addresses}{{% additional braces for segregating \footnotesize
  \bigskip
  \footnotesize

Surena Hozoori, \textsc{Department of Mathematics, University of Rochester, Rochester, NY 14627}\par\nopagebreak
  \textit{E-mail address}: \texttt{shozoori@ur.rochester.edu}
  
  }}
    \date{}
\begin{document}
\maketitle
\noindent
\begin{abstract}
We discuss a metric description of the divergence of a (projectively) Anosov flow in dimension~3, in terms of its associated expansion rates and give metric and contact geometric characterizations of when a projectively Anosov flow is Anosov. We then study the symmetries that the existence of an invariant volume form yields on the geometry of an Anosov flow, from various viewpoints of the theory of contact hyperbolas, Reeb dynamics and Liouville geometry, and give characterizations of when an Anosov flow is volume preserving, in terms of those theories. We finally use our study to show that the bi-contact surgery operations of Salmoiraghi \cite{salmoi,salmoi2} can be applied in an arbitrary small neighborhood of a periodic orbit of any Anosov flow. In particular, we conclude that the Goodman surgery of Anosov flows can be performed using the bi-contact surgery of \cite{salmoi2}.
\end{abstract}

%%%%%%%%%%%%%%%%%%%%%%%%%%%%%
\section{Introduction}\label{1}

For almost two decades after the introduction of Anosov flows in the early 1960s \cite{anosov0,anosov}, the only known examples of Anosov flows on three dimensional closed manifolds were based on either the suspension of Anosov diffeomorphisms of 2-torus, or the geodesic flows on the unit tangent space of hyperbolic surfaces. All such examples are orbit equivalent to an algebraic volume preserving flow, by their natural construction and a lot of interesting properties of Anosov flows were derived, assuming the existence of such invariant volume forms. However, the first examples of Anosov flows, which are not orbit equivalent to a volume preserving one, were constructed in 1980 by Franks and Williams~\cite{anomal}. Since then, understanding the relation between the existence of an invariant volume form and various aspects of Anosov dynamics has been studied from different viewpoints. In particular, from a topological viewpoint, such property is associated with the {\em transitivity} of an Anosov flow \cite{asaokaverj} and from a measure theoretic viewpoint, they correspond to ergodic Anosov flows \cite{anosov0,marg}. Moreover, many other dynamical aspects of such flows, including the regularity theoretical aspects, are well-studied in the literature (for instance, see \cite{sinai,hruder}).

Our goal in this paper, is to study the relation between the divergence of a flow and Anosovity, in the context of a larger class of dynamics, namely the class of {\em projectively Anosov flows}, and using the notion of {\em expansion rates of the invariant bundles}. These quantities measure the infinitesimal change of the length of vectors in the stable and unstable directions, and facilitate geometric understanding of Anosov flows. In particular, they play a significant role in the more recent {\em contact and symplectic geometric theory} of Anosov flows \cite{mitsumatsu,hoz3,salmoi}. Therefore, our study provides new perspective on the class of volume preserving Anosov flows, in terms of those geometries.

It is worth mentioning that although projectively Anosov flows have been previously studied in various contexts, such as foliation theory \cite{confoliations,asaokareg,noda,colin}, Riemannian geometry \cite{bl,blperrone,perrone,hoz2}, hyperbolic dynamics \cite{inv,hertz,puj1,puj2} and Reeb dynamics \cite{hoz1}, their primary significance for us is that they serve as bridge between Anosov dynamics and contact and symplectic geometry \cite{mitsumatsu} (see Section~\ref{2.2}), eventually yielding a complete characterization of Anosov flows in terms of such geometries \cite{hoz3}. We remark that such flows are also referred to in the literature, using other names including {\em conformally Anosov flows} or {\em flows with dominated splitting}.

\vskip0.5cm

\noindent\fbox{%
    \parbox{\textwidth}{%
\textbf{Assumptions:} In this paper, unless stated otherwise, we assume that $M$ is a closed connected oriented three manifold and $X$ is a non-vanishing $C^{1+}$ vector field, i.e. a $C^1$ vector field with Hölder continuous derivatives. We denote the $C^{1+}$ flow generated by $X$ by $\phi^t$. It is noteworthy that there are other vector fields and flows involved in this paper, for instance the Reeb vector fields of Theorem~\ref{cartanintro} and Theorem~\ref{liouintro}, for which we do not assume any regularity and in fact, are often only $C^0$ (also see Remark~\ref{reebreg}). We also assume the (projectively) Anosov flows to have transversely orientable invariant bundles. This is always achieved, possibly after lifting to a double cover of $M$. Moreover, we call any geometric quantity, which is differentiable in the direction of the flow, {\em $X$-differentiable}. The reader should consult \cite{palis} for the basics of the theory of flows on manifolds, and \cite{hflows} for the fundamentals of hyperbolic flows.        
    }%
}
\vskip0.5cm

We begin our study with a natural description of the divergence of a projectively Anosov flow in terms of its associated expansion rates of the invariant bundles, encapsulated in the following two theorems:

\begin{theorem}\label{metric1intro}
Let $X$ be the generator of a projectively Anosov flow on $M$ and $\Omega$ be some volume form which is $X$-differentiable. There exists a metric on $M$, such that $div_X\Omega=r_s+r_u$, where $r_s$ and $r_u$ are the expansion rates of the stable and unstable directions, respectively, measured by such metric.
\end{theorem}

\begin{theorem}\label{metric2intro}
Let $X$ be the generator of a projectively Anosov flow and $||.||$ some $X$-differentiable norm on $TM$, induced by a metric. Also, let $r_s$ and $r_u$ be the expansion rates of the stable and unstable bundles, measured by $||.||$. Then,

 (a) there exists a volume form $\Omega$ on $M$, which is $X$-differentiable and $div_X\Omega=r_s+r_u$,

 (b) for any $\epsilon>0$, there exists a $C^1$ volume form $\Omega^\epsilon$, such that $|div_X\Omega^\epsilon-(r_s+r_u)|<\epsilon$.
\end{theorem}

Although the above description of the divergence is hardly surprising, it accommodates the use of such relation from the viewpoint of differential and contact geometry. One immediate corollary~is 

\begin{corollary}\label{volproj}
Any projectively Anosov flow preserving some $C^0$ volume form is Anosov. In particular, any {\em contact} projectively Anosov flow (that is when a projectively Anosov flow preserves a transverse contact structure) is Anosov.
\end{corollary}

Although the above corollary is well known in the dynamical systems literature (for instance see \cite{threeflow}), it seems that this fact is unexpectedly left obscured in some other areas of research, most importantly when such flows appear in the Riemannian geometry literature. While contributing meaningfully to the related subjects, one can find many interesting results on the Riemannian geometry of contact projectively Anosov flows, ignoring that they are in fact Anosov (for instance, see \cite{bl,blperrone}).

It is well known that many important properties of (projectively) Anosov flows are independent of the norm involved in their definition. On the other hand, there are natural volume forms for such setting, induced from the underlying contact structures of these flows (see section~\ref{2.2}). It turns out that we can characterize the Anosovity of a projectively Anosov flow, in terms of the divergence of the flow being bounded by these volume forms in an appropriate sense (see Remark~\ref{inducedvol}).

\begin{theorem}\label{contcharintro}
Let $X$ be the generating vector field for a projectively Anosov flow. Then, the followings are equivalent:

(1) $X$ is Anosov.

(2) There exists a positive contact form $\alpha_+$, such that for some $\xi_-$, the pair $(\xi_-,\xi_+:=\ker{\alpha_+})$ is a supporting bi-contact structure and $-\alpha_+ \wedge d\alpha_+<(div_X\Omega^{\alpha_+})\Omega^{\alpha_+} < \alpha_+\wedge d \alpha_+.$

(3) There exists a negative contact form $\alpha_-$, such that for some $\xi_+$, the pair $(\xi_-:=\ker{\alpha_-},\xi_+)$ is a supporting bi-contact structure and $\alpha_- \wedge d\alpha_-<(div_X\Omega^{\alpha_-})\Omega^{\alpha_-} < -\alpha_-\wedge d \alpha_-.$
\end{theorem}

\begin{remark}
We remark that Theorem~\ref{metric1intro}-\ref{contcharintro} and Corollary~\ref{volproj} above hold for any $C^1$ flow (without the assumption of Hölder continuity for its derivative). However, for Theorem~\ref{contcharintro} in that case, we would need the approximation techniques developed in \cite{hoz3} to deal with $C^0$ weak stable and unstable bundles, since the Hölder continuity of the derivatives of the flow is needed to ensure such invariant plane fields are $C^1$, which is used in the proof of Theorem~\ref{contcharintro} for simplicity.
\end{remark}

Using our description of the divergence of an Anosov flow, we next study the geometric consequences of the existence of an invariant volume form for an Anosov flow, from various viewpoints. More precisely, Theorem~\ref{metric2intro} shows the symmetry of expansion and contraction in the unstable and stable directions, respectively, in the case of volume preserving Anosov flows and furthermore, thanks to the differentiability of the weak stable and unstable bundles in this case \cite{hruder,regular}, such symmetry behaves well, when translating the metric description of Anosov flows to the contact geometric one. We study such symmetry from the view point of {\em the theory of contact hyperbolas}, {\em Reeb dynamics} and {\em Liouville geometry}, giving various characterizations of volume preserving Anosov flows.

To begin with, we study volume preserving Anosov flows in terms of the theory of contact hyperbolas, developed by Perrone~\cite{perrone} (see Section~\ref{5.2} for definitions), as an analogue of the theory of {\em contact circles} by Geiges-Gonzalo~\cite{circle2,circle}. Moreover, we will see that these conditions are, in fact, equivalent to a purely Reeb dynamical description of volume preserving Anosov flows.

\begin{theorem}\label{cartanintro}
Let $\phi^t$ be a projectively Anosov flow on $M$. Then, the followings are equivalent:

(1) The flow $\phi^t$ is a volume preserving Anosov flow.

(2) There exists a supporting bi-contact structure $(\xi_-,\xi_+)$ and contact forms $\alpha_-$ and $\alpha_+$ for $\xi_-$ and $\xi_+$, respectively, such that $(\alpha_-,\alpha_+)$ is a $(-1)$-Cartan structure.

(3) There exists a supporting bi-contact structure $(\xi_-,\xi_+)$ and Reeb vector fields $R_{\alpha_-}$ and $R_{\alpha_+}$ for $\xi_-$ and $\xi_+$, respectively, such that $R_{\alpha_-}\subset \xi_+$ and $R_{\alpha_+}\subset \xi_-$.
\end{theorem}

To the best of our knowledge, the only known examples of taut contact hyperbolas, except an explicit example constructed on $\mathbb{T}^3$, are achieved using the symmetries of Lie manifolds, giving examples which are compatible with algebraic Anosov flows \cite{perrone}. However, Theorem~\ref{cartanintro} shows that we can also construct examples of taut contact hyperbolas on hyperbolic manifolds, thanks to the construction of an infinite family of contact Anosov flows on hyperbolic manifolds by Foulon-Hasselblatt \cite{foulon}, as well as many examples on toroidal manifolds \cite{bbu}. We note that it is not known if any specific manifold can admit infinitely many distinct Anosov flows, while there are at most finitely many contact Anosov flows on any manifold up to orbit equivalence~\cite{mann}. This gives  a partial answer to the classification problem posed in the final remark of \cite{perrone}.

\begin{corollary}
There exist infinitely many hyperbolic manifolds which admit a $(-1)$-Cartan structure (and in particular, a taut contact hyperbola).
\end{corollary}

Moreover, we study volume preserving Anosov flows from the perspective of Liouville geometry. The construction of exact symplectic 4-manifold for a general Anosov 3-flow is done by the author in \cite{hoz3}. However, we observe that such construction is significantly simplified in the presence of an invariant volume form (the case previously studied by Mitsumatsu \cite{mitsumatsu}). In fact, after a canonical reparametrization of a volume preserving Anosov flow, we show that we can improve the relation between such flows and both the underlying Reeb dynamics of Theorem~\ref{cartanintro}, as well as the Liouville geometry associated with the corresponding exact symplectic 4-manifold. We call such reparametrization {\em the Liouville reparametrization of a volume preserving Anosov flow} (see Section~\ref{5.3} for definitions).

\begin{theorem}\label{liouintro}
Let $X$ be the generating vector field of a  volume preserving Anosov flow. If $X_L$ is the generating vector field for the Liouville reparametrization of the flow, the following holds:

(1) The flow generated by $X_L$ preserves the transverse plane field $\langle R_{\alpha_-},R_{\alpha_+} \rangle$, where $R_{\alpha_-}$ and $R_{\alpha_+}$ are the Reeb vector fields of Theorem~\ref{cartanintro} (2);

(2) The pair $(M,X_L)$ can be extended to a Liouville structure $([-1,1]\times M,Y)$, such that $([-1,1]\times M,Y)\big|_{\{0\}\times M}=(M,X_L).$
\end{theorem}

\begin{remark}
It is important to note that in the above theorem, the plane field generated by $R_{\alpha_-}$ and $R_{\alpha_+}$ is only continuous in the general case, and is $C^1$, if and only if the Liouville reparametrization generated by $X_L$ is contact or a suspension flow (this is done for $C^2$ flows in \cite{zyg}, but similar proof should work in the $C^{1+}$ category). In particular, this gives a bi-contact geometric way of distinguishing which Anosov flows are contact or a suspension flow. They are exactly the ones, where the constructed Reeb vector fields $R_{\alpha_-}$ and $R_{\alpha_+}$ in Theorem~\ref{liouintro} are $C^1$.
\end{remark}

At the end, we discuss the applications of our study to the surgery theory of Anosov flows. Surgery theory has been a very important part of the geometric theory of Anosov flows from the early days. Various Dehn-type surgery operations, including Handel-Thurston \cite{handel}, Goodman~\cite{goodman}, Fried~\cite{fried} or Foulon-Hasselblatt \cite{foulon} surgeries, have helped construction of new examples of Anosov flows, answering historically important questions. These include the first examples of Anosov flows on hyperbolic manifolds \cite{goodman}, the construction of infinitely many {\em contact} Anosov flows on hyperbolic manifolds \cite{foulon} or the first (non-trivial) classification of Anosov flows on hyperbolic manifolds \cite{binyu}.

Recently, Salmoiraghi \cite{salmoi,salmoi2} has introduced two novel bi-contact geometric surgery operations of (projectively) Anosov flows, which contribute towards the contact geometric theory of Anosov flows (see \cite{mitsumatsu,hoz3,bowden} for instance) and the related surgery theory, reconstructing previously known surgery operations of Foulon-Hasselblatt and Handel-Thurston. These surgeries are applied in the neighborhood of a {\em Legendrian-transverse knot}, i.e. a knot which is {\em Legendrian} (tangent) for one of the underlying contact structures in the {\em supporting bi-contact} and transverse for the other one (see Section~\ref{2.2}). One of these surgery operations is done by cutting the manifold along an annulus tangent to the flow and the other one is based on a transverse annulus. However, the relation to Goodman surgery, which is one of the most significant surgery operations on Anosov flows, and is applied in the neighborhood of a periodic orbit of such flow, relies on one condition. That requires being able to push a periodic orbit to a Legendrian-transverse knot. Salmoiraghi observes that this is possible for the unit tangent space of hyperbolic surfaces \cite{salmoi} and furthermore, shows that if such condition is satisfied, the Goodman surgery can be reconstructed, using the bi-contact surgery on a transverse annulus (in fact, he generalizes such operation to projectively Anosov flows) \cite{salmoi2}. We show that such condition can be satisfied for any Anosov flow, by choosing a norm which yields constant divergence on a given periodic orbit of the flow, giving an affirmative answer to the question posed in \cite{salmoi}. This takes us one step closer to a contact geometric surgery of Anosov flows, unifying the previously introduced operations (it is noteworthy that the equivalence of Fried and Goodman surgeries has been recently shown for transitive Anosov flows \cite{shannon}, which is conjecturally true in the general case, hence resulting in the use of the term Goodman-Fried surgery in the literature).

\begin{theorem}\label{surgeryintro}
Let $\phi^t$ be an Anosov flow. Given any periodic orbit $\gamma_0$, there exists a supporting bi-contact structure $(\xi_-,\xi_+=\ker{\alpha_+})$, such that we have $R_{\alpha_+}\subset \xi_-$ in a regular neighborhood of $\gamma_0$. Therefore, there exists an isotopy $\{\gamma_t\}_{t\in [0,1]}$, which is supported in an arbitrary small neighborhood of $\gamma_0$, and $\gamma_t$ is a  Legendrian-transverse knot for any $0<t\leq 1$.
\end{theorem}

%It is noteworthy that since the introduction of Goodman \cite{goodman} and Fried \cite{fried} surgeries in the early 1980s, it has been a folklore theorem that the two operations produce orbit equivalent flows, when applied to a periodic orbit of an Anosov flow and the term Fried-Goodman surgery has been commonly used in the literature. However, such claim has been formally proven only recently and in the category of transitive Anosov flows \cite{shannon}. In \cite{salmoi2}, Salmoiraghi claims that in an upcoming paper, he uses Theorem~\ref{surgeryintro} to prove this in the general case.

\begin{corollary}
The bi-contact surgeries of Salmoiraghi \cite{salmoi,salmoi2} can be applied in an arbitrary small neighborhood of a periodic orbit of any  Anosov flow. In particular, the bi-contact surgery of \cite{salmoi2} reconstructs the Goodman surgery.
\end{corollary}

\vskip0.25cm

In Section~\ref{2}, we review some basic notions in Anosov dynamics and the expansion rates, as well as their connection to contact geometry. In Section~\ref{3}, we describe the divergence of a (projectively) Anosov flow in terms of its associated expansion rates. In Section~\ref{4}, we discuss some useful interplays between the contact geometry of Anosov flows and various volume forms on a three manifold, giving a contact geometric characterization of Anosovity, based on divergence, In Section~\ref{5}, we study the symmetries that the existence of an invariant volume form implies on the geometry of an Anosov flow, from various viewpoints of the theory of contact hyperbolas, Reeb dynamics and Liouville geometry. Finally, Section~\ref{6} is devoted to discussing the applications of our study to bi-contact surgeries.

\vskip0.5cm
\textbf{ACKNOWLEDGEMENT:} I want to thank my advisor, John Etnyre, for constant support and encouragement. I am also very grateful to Meysam Nassiri, Federico Salmoiraghi, Domenico Perrone, Thomas Barthelmé, Masayuki Asaoka and Steven Hurder for helpful conversations. Finally, I owe a debt of gratitude to the reviewers for thoughtful examination and insightful suggestions. The author was partially funded by NSF grant DMS-1906414.

%%%%%%%%%%%%%%%%%%%%%%%%%%
%%%%%%%%%%%%%%%%
\section{Background}\label{2}

In this section, we bring the necessary background for the main results. First, we review some basics about Anosov flows in dimension 3 and their generalization to projectively Anosov flows. Then, we discuss the connection of such flows to contact geometry. This is not, by any means, a thorough treatment and one should consult references like \cite{hflows,bart} and \cite{geiges} on these subjects for a more complete perspective.

%%%%%%%%%%%%%%%%%%%%
\subsection{(Projectively) Anosov flows and the associated expansion rates}\label{2.1}

Anosov flows in dimension 3 are non-singular flows, whose action on the tangent space of the ambient manifold exhibits exponential expansion and contraction in two distinct transverse directions. More formally,

\begin{definition}\label{anosov}
Let $\phi^t$ be the flow, generated by the non-vanishing $C^1$ vector field $X$. We call $\phi^t$ {\em Anosov}, if there exists a continuous invariant splitting $TM\simeq E^{ss} \oplus E^{uu} \oplus \langle X\rangle$, such that for some positive constant $C$ and a norm $||.||$, we have
$$||\phi^t_*(u)||\leq e^{-Ct}||u||\ \ \ \text{and}\ \ \ ||\phi^t_*(v)||\geq e^{Ct}||v||,$$
for any $u\in E^{ss}$ and $v\in E^{uu}$. We call $E^{ss}$ and $E^{uu}$ strong stable and unstable directions, respectively.
\end{definition}

The classical examples of such flows are the geodesic flows on the unit tangent bundle of hyperbolic surfaces and the suspension of Anosov diffeomorphisms of torus. However, various surgery operations on Anosov flows have yielded many more examples, including infinitely many examples on hyperbolic manifolds. These include surgeries of Handel-Thurston \cite{handel}, Fried \cite{fried}, Goodman~\cite{goodman}, Foulon-Hassleblatt \cite{foulon} and more recently, the {\em bi-contact geometric} surgeries introduced by Salmoiraghi \cite{salmoi,salmoi2}, which manage to reproduce, up to orbit equivalence, the previous operations in many cases.

\begin{remark}\label{poincare}
We remark that the Anosovity of a non-singular flow can be determined by its action on the normal bundle of the direction of the flow. More precisely, a flow $\phi^t$, generated by the non-vanishing vector field $X$, induces a flow on $TM/\langle X \rangle$ via $\pi:TM\rightarrow TM/\langle X \rangle$, usually called {\em the induced Poincaré linear flow}. It is a classical result in dynamical systems by Doering \cite{doering} that a flow is Anosov, if and only if the induced Poincaré linear flow admits a {\em hyperbolic splitting}. That is, there exists a continuous splitting of the normal bundle $TM/\langle X\rangle\simeq E^s \oplus E^u$, which is invariant under the induced Poincaré linear flow and with respect to some norm, the action of such flow on $E^u$ and $E^s$ is exponentially expanding and contracting, respectively.
\end{remark}

One important generalization of Anosov flows, which bridges Anosov dynamics to contact geometry is the following:

\begin{definition}\label{pA}
Let $\tilde{\phi}^t$ be the Poincaré linear flow as in Remark~\ref{poincare}. We call $\phi^t$ {\em projectively Anosov}, if there exists a continuous invariant splitting $TM/\langle X\rangle\simeq E^{s} \oplus E^{u}$, such that for some positive constant $C$ and a norm $||.||$, we have
$$||\tilde{\phi}^t_*(v)||/||\tilde{\phi}^t_*(u)||\leq e^{Ct}||v||/||u||,$$
for any $u\in E^{s}$ and $v\in E^{u}$. We call $E^s$ and $E^u$ stable and unstable directions, respectively.
\end{definition}

In other words, a flow is projectively Anosov, if its induced Poincaré linear flow admits a {\em dominated splitting}. That is, a continuous and invariant splitting into two line bundles, on which the action of the flow is relatively expanding in one direction with respect to the other. 

\vskip0.25cm
\noindent\fbox{%
    \parbox{\textwidth}{%
Abusing notation, we also refer to $\pi^{-1}(E^s)$ and $\pi^{-1}(E^u)$, which are a priori $C^0$ two dimensional sub bundles of $TM$, by $E^s$ and $E^u$, respectively, and call them {\em the weak stable and unstable bundles}, respectively (see Remark~\ref{shifting}).
}}
\vskip0.25cm

In the case when a projectively Anosov flow is Anosov, the weak stable and unstable bundles are known to be $C^1$ \cite{regular}.
It is worth mentioning that that unlike the Anosov case (as discussed in Remark~\ref{poincare}), in the case of projectively Anosov flows, the splitting of the normal bundle $TM/\langle X\rangle$ cannot necessarily be lifted to the tangent space $TM$ \cite{noda}.

Although it is not a priori obvious if such class of dynamics is strictly larger the class of Anosov flows, we know that projectively Anosov flows are abundant. See \cite{mitsumatsu,confoliations} for examples on torus bundles, \cite{bowden} for non-Anosov examples on hyperbolic manifolds and \cite{adn} for a more general construction.

In order to build the bridge from the above definitions to the world of differential and contact geometry, it is very useful for us to measure the infinitesimal expansion or contraction of the length of vectors in the invariant bundles. Note that without loss of generality, we can assume the norm involved in the definition of (projectively) Anosov flows is $C^\infty$.
%Note that, after integrating the norm in the above definitions in the direction of the flow, i.e. replacing $||.||$ with $\frac{1}{T}\int_0^T \phi^{t*}||.||dt$, one can assume that the norm satisfying the (projective) Anosovity condition is {\em $X$-differentiable}, i.e. differentiable in the direction of the flow. 

\begin{definition}
Using the above notation and considering $TM/\langle X\rangle\simeq E^{s} \oplus E^{u}$, let $\tilde{e}_s\in E^s$ and $\tilde{e}_u\in E^u$ be the unit vectors with respect to some $X$-differentiable norm $||.||$ on $TM/\langle X\rangle$. We call
$$r_s:=\frac{\partial}{\partial t} \ln{||\tilde{\phi}_*^t (\tilde{e}_s)||}\bigg|_{t=0}\ \  \text{and} \ \   r_u:=\frac{\partial}{\partial t} \ln{||\tilde{\phi}_*^t (\tilde{e}_u)||}\bigg|_{t=0}$$ {\em the expansion rates of the stable and unstable directions}, respectively, with respect to $||.||$.
\end{definition}

We remark that similar notions have been previously utilized in the study of various aspects of Anosov flows \cite{simic,regular,sharp}.

\begin{remark}\label{shifting}
Note that the norm used in the definition of expansion rates is defined on the normal bundle $TM/\langle X\rangle$. However, it is easy to describe these quantities based on the tangent bundle $TM$. First, given a norm on $TM/\langle X\rangle$, consider some metric on this vector bundle, which induces the norm. Notice that given any transverse $C^1$ plane field $\eta$,  there exists a natural isomorphism $\eta\simeq TM/\langle X\rangle$ via the projection $\pi:TM\rightarrow \eta\simeq TM/\langle X\rangle$. Therefore, a Riemannian metric on $TM/\langle X\rangle$ induces a Riemannian metric on $\eta$, which can be naturally extended to $TM$, by letting $||X||=1$ and $X\perp \eta$, where $||.||$ is the norm on $TM$, induced from such metric. If $\tilde{e}_s \in E^s \subset TM/\langle X\rangle$ and $\tilde{e}_u\in E^u\subset TM/\langle X\rangle$ are the unit vector fields, their image $e_s\in \pi^{-1}(E^s)\cap \eta$ and $e_u\in \pi^{-1}(E^u)\cap \eta$, under the isomorphism are unit vector fields with respect to the norm induced on $TM$. It is easy to compute
$$\mathcal{L}_X e_s=-r_se_s +q_s X\ \ \ \text{and}\ \ \ \mathcal{L}_X e_u=-r_ue_u +q_u X,$$
where $q_s$ and $q_u$ are real functions, depending on our choice of $\eta$. Note that we will have $q_s=q_u=0$, when $\eta$ is preserved by $X$ (see Section~3 of \cite{hoz3} for more thorough discussion). Since we are assuming $\eta$ to be $C^1$, this only happens for an Anosov flow, when it is contact or a suspension flow~\cite{zyg}.

Note that this remark also justifies the abuse of notation we adopted in the beginning of this section, i.e. referring to $\pi^{-1}(E^s)$ by $E^s$ (and similarly for $E^u$). To be more clear, the correspondence between the metrics on $TM$ and $TM/\langle X\rangle$ discussed above allows us to talk about the unit vector $e_s\in E^s$ and therefore, the expansion rate of the stable direction unambiguously, whether we consider $E^s\subset TM$ or $E^s \subset TM/\langle X \rangle$. The same holds for $e_u\in E^u$ and the expansion rate of the unstable direction.
\end{remark}

\begin{remark}\label{formshifting}
Alternatively, one can characterize the expansion rates in terms of differential forms. Consider a metric on $TM/\langle X \rangle$ and define $\tilde{\alpha_s}$ as a differential form on $TM/\langle X \rangle$, by letting $\ker{\tilde{\alpha_s}}=E^u\subset TM/\langle X \rangle$ and $\tilde{\alpha_s}(\tilde{e}_s)=1$, where $\tilde{e}_s\in E^s\subset TM/\langle X \rangle$ is a unit vector field with respect to the chosen metric. Similarly, define $\tilde{\alpha}_u$ and easily compute
$$\mathcal{L}_X \tilde{\alpha}_s=r_s\tilde{\alpha}_s\ \ \ \text{and}\ \ \ \mathcal{L}_X \tilde{\alpha}_u=r_u\tilde{\alpha}_u,$$
where $r_s$ and $r_u$ are the expansion rates of the stable and unstable directions with respect to such metric.

As we will see in the remainder of the paper, it is usually desired to work with differential forms on $TM$. Therefore, we can define $\alpha_s:=\pi^*\tilde{\alpha}_s$ and $\alpha_u:=\pi^*\tilde{\alpha}_u$, where $\pi:TM\rightarrow TM/\langle X\rangle\simeq \eta$ is the natural projection, and note that
$$\mathcal{L}_X \alpha_s=r_s\alpha_s\ \ \ \text{and}\ \ \ \mathcal{L}_X \alpha_u=r_u\alpha_u.$$

In fact, the above expressions can be taken as the definition of the expansion rates $r_s$ and $r_u$. More precisely, if we take the unit vector fields $e_s\in E^s\cap\eta\subset TM$ and $e_u\in E^u\cap\eta\subset TM$, induced after choosing a transverse plane field $\eta$ as in Remark~\ref{shifting}, and define the differential form $\alpha_s$ by letting $\ker{\alpha_s}=E^u\oplus \langle X\rangle$ and $\alpha_s(e_s)=1$, this matches the definition above. The same holds for $\alpha_u$. This means that as in Remark~\ref{shifting}, the choice of the transverse plane field $\eta$ does not affect the geometry of expansion in terms of differential forms on $TM$.
\end{remark}

Not surprisingly, the (relative) exponential expansion for (projectively) Anosov flows can be easily characterized in terms of such expansion rates (see \cite{hoz3} for more details and proofs).

\begin{proposition}\label{rateca}
Let $X$ be a projectively Anosov vector field and $r_s$ and $r_u$, the expansion rates of stable and unstable directions, respectively, with respect to any Riemannian metric, satisfying the metric condition of Definition~\ref{pA}, which is $X$-differentiable, then 
$$r_u-r_s>0.$$
\end{proposition}

While Proposition~\ref{rateca} expresses the {\em relative} expansion in the unstable direction with respect to the stable direction, the following proposition realizes when we have {\em absolute} expansion and contraction in those directions, which by Doering's result \cite{doering} yields Anosovity (see Remark~\ref{poincare}).

\begin{proposition}\label{rateanosov}
Let $X$ be a projectively Anosov vector field and $r_s$ and $r_u$. Then $X$ is Anosov, if and only if, with respect to some Riemannian metric, we have $$r_u>0>r_s.$$
\end{proposition}

%%%%%%%%%%%%%%
\subsection{Relation to (bi-)contact geometry}\label{2.2}

Recall that a $C^1$ 1-form $\alpha$ is a {\em contact form} on $M$, if
$\alpha \wedge d\alpha$ is a non-vanishing volume form on $M$.
If $\alpha \wedge d\alpha>0$ (compared to the orientation on $M$), we call $\alpha$ a {\em positive} contact form and otherwise, a {\em negative} one. We call the $C^1$ plane field $\xi:=\ker{\alpha}$ a (positive or negative) {\em contact structure} on $M$. Notice that by the Frobenuis theorem, contact structures can be thought of as {\em maximally non-integrable} $C^1$ plane fields, i.e. the extreme opposite of foliations.

For example, $\xi_{std}:=\ker{dz-ydx}$ ($\xi_{std}:=\ker{dz+ydx}$) is called the {\em standard} positive (negative) contact structure on $\mathbb{R}^3$, while $\xi_n:=\ker{\{\cos{2\pi n z}dx-\sin{2\pi n} dy \}}$ on $\mathbb{T}^3=\mathbb{R}^3/\mathbb{Z}^3$ gives an infinite family of distinct positive (negative) contact structures, when $n\in\mathbb{Z}>0$ ($n<0$).

Although we do not go towards the topological aspects of contact structures in this paper, it is worth mentioning that positive (negative) contact structures do not have any local invariant, thanks to the Darboux theorem that states that any two positive (negative) contact structures are locally {\em contactomorphic} (i.e. locally look like the standard model on $\mathbb{R}^3$). Furthermore, Gray's theorem shows that any homotopy of a contact structure through contact structures is induced by an isotopy of the ambient manifold.

Associated to any contact structure, there is an important class of flows, which we will utilize in this paper. Given any contact form $\alpha$ for a contact structure $\xi:=\ker{\alpha}$, there exists a unique vector field $R_{\alpha}$, satisfying

$$d\alpha(R_{\alpha},.)=0\ \ \ \text{and}\ \ \ \alpha(R_\alpha)=1.$$

Such vector field is called a {\em Reeb vector field} and it is easy to check $\mathcal{L}_{R_\alpha}\alpha=0$. This implies $\mathcal{L}_{R_\alpha}\alpha \wedge d\alpha=0$. In particular Reeb vector fields are volume preserving. Furthermore, Reeb vector fields preserve $\xi$, and are transverse to the underlying contact structure. It is easy to observe that conversely, given a contact structure, any transverse vector field preserving the contact structure $\xi$ is a Reeb vector field for an appropriate choice of contact form.

A natural and well studied interplay of contact geometry and Anosov dynamics happens when a Reeb vector field is Anosov, i.e. the case of {\em contact Anosov flows} (see for instance \cite{foulon}). However, in this paper, we are interested in a more general relation between the two theories, thanks to the following proposition, first observed by Mitsumatsu \cite{mitsumatsu} and Eliashberg-Thurston \cite{confoliations}, which characterizes projectively Anosov flows in terms of contact geometry. We remind the reader that as mentioned in the introduction of the paper, we are assuming the underlying manifold to be oriented and the invariant bundles for the projectively Anosov flows to be transversely orientable.

\begin{proposition}\label{equiv}
Let $X$ be a non-vanishing $C^1$ vector field on $M$. Then, $X$ generates a projectively Anosov flow, if and only if, there exist positive and negative contact structures, $\xi_+$ and $\xi_-$ respectively, which are transverse and $X\subset \xi_+ \cap \xi_-$.
\end{proposition}

In other words, considering a projectively Anosov flow, the bi-sectors of $E^s$ and $E^u$ can be seen to be a pair of positive and negative contact structures, possibly after a perturbation to make them $C^1$. And conversely, any vector field directing the intersection of such transverse pair is projectively Anosov.

We note that the above proposition also shows that the (periodic) orbits of a projectively Anosov flows are {\em Legendrian} (knots), i.e. tangent, for both of the underlying contact structures.

Using Proposition~\ref{equiv}, we can easily give examples of non-Anosov projectively Anosov flows. For instance, $(\xi_m:=\ker{\{dz+\epsilon (cos{2\pi m z}dx-\sin{2\pi m} dy)\}},\xi_n:=\ker{\{dz+\epsilon'(cos{2\pi n z}dx-\sin{2\pi n} dy) \}})$ is a pair of positive and negative transverse contact structures, whenever $m<0<n$ are integers and $\epsilon\neq \epsilon'$. Therefore, any flow, whose generating vector field lies in the intersection $\xi_m\cap \xi_n$, is a projectively Anosov flow on $\mathbb{T}^3=\mathbb{R}^3/\mathbb{Z}^3$ by Proposition~\ref{equiv}. Note that there are no Anosov flows on $\mathbb{T}^3$.

We call such pair of transverse negative and positive contact structures $(\xi_-,\xi_+)$ a {\em bi-contact structure}, {\em supporting} the underlying projectively Anosov flow. It turns out that by enriching a bi-contact structure with additional contact geometric structures, one can also characterize Anosov flows, purely in terms of contact geometry \cite{hoz3}.

%%%%%%%%%%%%%%%%%%%%%%%%%%%%%
\section{Divergence and the expansion rates}\label{3}

In this section, we show that the divergence of a projectively Anosov flow with respect to the (a priori $C^0$) volume form, which is induced from any norm satisfying the relevant definition, can naturally be characterized in terms of the expansion rates of the stable and unstable directions. We then give approximation results for volume forms with higher regularity.

\begin{theorem}\label{metric1}
Let $X$ be the generator of a projectively Anosov flow on $M$ and $\Omega$ be some volume form which is $X$-differentiable. There exists a metric on $M$, such that $div_X\Omega=r_s+r_u$, where $r_s$ and $r_u$ are the expansion rates of the stable and unstable directions, respectively, measured by the metric.
\end{theorem}

\begin{proof}
Choose a $C^\infty$ transverse plane field $\eta$ and let $\alpha_X$ be a $C^1$ 1-form such that $\alpha_X(\eta)=0$ and $\alpha_X(X)=1$. Furthermore, choose some contact form $\tilde{\alpha}_+$, so that $(\xi_-,\xi_+:=\ker{\tilde{\alpha}_+})$ is a supporting bi-contact structure for $X$, for some negative contact structure $\xi_-$.

We can write $\tilde{\alpha}_+=\tilde{\alpha}_u-\tilde{\alpha}_s$, where $\tilde{\alpha}_u\big|_{E^s}=\tilde{\alpha}_s\big|_{E^u}=0$. Notice that $\tilde{\alpha}_u$ and $\tilde{\alpha}_s$ are $C^0$ 1-forms, which are $X$-differentiable, since $\tilde{\alpha}_+$ is $C^1$ and the projection resulting in such decomposition is $X$-differentiable.

Since $\tilde{\alpha}_s \wedge \tilde{\alpha}_u \wedge \alpha_X$ is a volume form on $M$, there exists a positive function $f:M\rightarrow \mathbb{R}^{+}$, such that $|\Omega|=|\alpha_s \wedge \alpha_u \wedge \alpha_X|$, where $\alpha_s=f\tilde{\alpha}_s$ and $\alpha_u=f\tilde{\alpha}_u$.

Finally, we can define the norm $||.||$ with $||X||=||e_s||=||e_u||=1$, where $e_s\in E^s\cap\eta$, $e_u\in E^u\cap\eta$, $|\alpha_s(e_s)|=|\alpha_u(e_u)|=1$ and $(e_s,e_u,X)$ is a an oriented basis for $TM$. Notice that by construction, $||.||$ is $X$-differentiable.

Letting $r_s$ and $r_u$ be the expansion rates of the stable and unstable directions, respectively, we can compute

$$(\mathcal{L}_X\Omega)\  (e_s,e_u,X)=-\Omega([X,e_s],e_u,X)-\Omega(e_s,[X,e_u],X)=(r_s+r_u)\ \Omega(e_s,e_u,X),$$
completing the proof.
\end{proof}

\begin{corollary}
Any projectively Anosov flow preserving some $C^0$ volume form is Anosov. In particular, any contact projectively Anosov flow is Anosov.
\end{corollary}

\begin{proof}
Note that any preserved $C^0$ volume form is $X$-differentiable ($\mathcal{L}_X\Omega=0$). Therefore, Theorem~\ref{metric1} and Proposition~\ref{rateca} imply $r_s<0<r_u$, which guarantees Anosovity.
\end{proof}

In Theorem~\ref{metric1}, we show that given a volume form, we can find a norm on $TM$ such that the sum of its associated expansion rates equals the divergence of our volume form. The following theorem yields the inverse construction. That is, given a norm on $TM$, we can construct a volume form whose divergence is given by the sum of the expansion rates induced by our metric.

\begin{theorem}\label{metric2}
Let $X$ be the generator of a projectively Anosov flow and $||.||$ some $X$-differentiable norm on $TM$, induced by a metric. Also, let $r_s$ and $r_u$ be the expansion rates of the stable and unstable bundles, measured by $||.||$. Then,

 (a) there exists a volume form $\Omega$ on $M$, which is $X$-differentiable and $div_X\Omega=r_s+r_u$,

 (b) for any $\epsilon>0$, there exists a $C^1$ volume form $\Omega^\epsilon$, such that $|div_X\Omega^\epsilon-(r_s+r_u)|<\epsilon$.
\end{theorem}

\begin{proof}
(a) Choose a $C^\infty$ transverse plane field $\eta$. Via $\eta$, the norm involved in the definition of the expansion rates will induce a norm $||.||$ on $TM$ (see Remark~\ref{shifting}). Define $e_s\in E^s\cap\eta$ and $e_u\in E^u\cap\eta$, so that $||e_s||=||e_u||=1$ and $(e_s,e_u,X)$ is an oriented basis for $TM$. Finally, define the 1-forms $\alpha_s$, $\alpha_u$ and $\alpha_X$ so that $\alpha_s\big|_{E^u}=\alpha_u\big|_{E^s}=\alpha_X\big|_{\eta}=0$ and $\alpha_s(e_s)=\alpha_u(e_u)=\alpha_X(X)=1$. 

Letting $\Omega:=\alpha_s\wedge\alpha_u\wedge\alpha_X$, it is easy to see $div_X\Omega=r_s+r_u$, as in Theorem~\ref{metric1}.
\vskip0.25cm

(b) Let $\Omega$ be the volume form constructed in part (a) and $\Omega^\infty$ be any $C^\infty$ volume form on $M$. There exist a $X$-differentiable function $f:M\rightarrow \mathbb{R}^+$, such that $\Omega=f\Omega^\infty$. Notice that
$$\mathcal{L}_X\Omega=(X\cdot f)\Omega^\infty +f\mathcal{L}_X\Omega^\infty=(div_X\Omega)\Omega.$$

As in Lemma~4.3 in \cite{hoz3}, there exists a $C^1$ function $f^\epsilon$ such that $|f^\epsilon-f|$ and $|X\cdot f^\epsilon -X\cdot f|$ are arbitrary small. Therefore, letting $\Omega^\epsilon:=f^\epsilon \Omega^\infty$ and computing
$$\mathcal{L}_X\Omega^\epsilon=(X\cdot f^\epsilon)\Omega^\infty +f^\epsilon\mathcal{L}_X\Omega^\infty=(div_X\Omega^\epsilon)\Omega^\epsilon,$$
we confirm that $div_X\Omega^\epsilon$ can be taken to be arbitrary close to $div_X\Omega=r_s+r_u$.
\end{proof}

%%%%%%%%%%%%%%
\section{A contact geometric characterization of Anosovity\\ based on divergence}\label{4}

In this section, we show that we can use the volume forms, naturally coming from the underlying contact structures (see Section~\ref{2.2}), to give necessary and sufficient conditions for Anosovity of a projectively Anosov flow, which is independent of the metric and utilizes the expansion rates. %We note that to go from the natural setting of projectively Anosov flows, with a $C^0$ splitting which is a priori solely differentiable along the flow, to the contact geometric setting, which involves $C^1$ geometric objects, we need subtle approximation techniques. These techniques were also used by the author in \cite{hoz3}.

The following remark shows that a given contact form for one of the underlying contact structures of a projectively Anosov flow, induces a natural volume form, as well as a norm, with respect to which, we can compute the expansion rates.

\begin{remark}\label{inducedvol}
Notice that if $(\xi_-:=\ker{\alpha_-},\xi_+:=\ker{\alpha_+})$ is a supporting bi-contact structure for a projectively Anosov flow, $\alpha_+$ ($\alpha_-$) naturally defines two volume forms on $M$, one being the {\em contact volume form}, i.e. $\alpha_+\wedge d\alpha_+$ ($\alpha_-\wedge d\alpha_-$). Additionally, we can uniquely write $\alpha_+=\alpha_u-\alpha_s$ ($\alpha_-=\alpha_u+\alpha_s$), where $\alpha_u$ and $\alpha_s$ are continuous 1-forms, such that $\ker{\alpha_u}=E^s\subset TM$, $\ker{\alpha_s}=E^u\subset TM$, $\alpha_s(e_s)>0$ and $\alpha_u(e_u)>0$. Here, $e_s\in E^s\cap \eta$ and $e_u\in E^u\cap\eta$ for some $C^1$ transverse plane field $\eta$, such that $(e_s,e_u,X)$ is an oriented basis for $TM$. This induces the positive volume form $\Omega^{\alpha_+}:=\alpha_s\wedge\alpha_u \wedge \alpha_X$ ($\Omega^{\alpha_-}:=\alpha_s\wedge\alpha_u \wedge \alpha_X$), where $\alpha_X$ is any 1-form satisfying $\alpha_X(X)=1$.

Furthermore, $\alpha_+$ ($\alpha_-$) defines a norm on $TM/\langle X\rangle$, using the above argument and  the natural one-to-one correspondence between the differential forms on $TM/\langle X \rangle$ and the differential forms on $TM$ whose kernels include $X$.
%letting $||\tilde{e}_s||=||\tilde{e}_u||=1$, where $e_s\in E^s$ and $e_u\in E^u$ are vectors in $TM/\langle X\rangle$, satisfying $\pi^*\alpha_s(\tilde{e}_s)=\pi^*\alpha_u(\tilde{e}_u)=1$. With respect to such norm, we can measure the expansion rates of the underlying flow.

Notice that the definition of $\Omega^{\alpha_+}$ ($\Omega^{\alpha_-}$) above does not depend on the choice of $\eta$ (see Remark~\ref{formshifting}) and the oriented basis $(e_s,e_u,X)$. In particular, choosing an oriented basis of the form $(e_u,e_s,X)$ only changes our convention for splitting $\alpha_+$ ($\alpha_-$), i.e. we would need to write $\alpha_+=\alpha_u+\alpha_s$ ($\alpha_-=\alpha_u-\alpha_s$) in that case. However, these choices will not effect the constructed positive volume form $\Omega^{\alpha_+}$ ($\Omega^{\alpha_-}$).
\end{remark}

Here, we bring two lemmas, which will simplify the computations in the proof of Theorem~\ref{contchar}.

\begin{lemma}\label{contcomp}
Let $\alpha_+$ and $\alpha_-$ be positive and negative contact forms, such that $(\xi_-:=\ker{\alpha_-},\xi_+:=\ker{\alpha_+})$ is a supporting bi-contact structure for the projectively Anosov flow generated by $X$. Moreover, let $\Omega^{\alpha_+}$ ($\Omega^{\alpha_-}$) be the volume form, and $r_u^+$ and $r_s^+$ ($r_u^-$ and $r_s^-$) be the expansion rates, induced by $\alpha_+$ ($\alpha_-$) as in Remark~\ref{inducedvol}. Then,
$$\alpha_+\wedge d\alpha_+ =(r_u^+-r_s^+)\Omega^{\alpha_+} \ \bigg(\alpha_-\wedge d\alpha_-=-(r_u^--r_s^-)\Omega^{\alpha_-}\bigg).$$
\end{lemma}

\begin{proof}
Let $e_s\in E^s$ and $e_u\in E^u$ be the unit vector fields on $TM$, defined as in Remark~\ref{shifting}.
$$\alpha_+\wedge d\alpha_+=\bigg\{\alpha_+(e_s)d\alpha_+(e_u,X)-\alpha_+(e_u)d\alpha_+(e_s,X)\bigg\}\Omega^{\alpha_+}$$
$$=\bigg\{-\alpha_+(e_s)\alpha_+([e_u,X])+\alpha_+(e_u)\alpha_+([e_s,X])\bigg\}\Omega^{\alpha_+}=(r^+_u-r^+_s)\Omega^{\alpha_+}.$$

Similar computation for $\alpha_-$ finishes the proof.
\end{proof}
Note that Theorem~\ref{metric1} also yields:
\begin{lemma}\label{volcomp}
With the notation of Lemma~\ref{contcomp},
$$\mathcal{L}_X\Omega^{\alpha_+}=(r_u^++r_s^+)\Omega^{\alpha_+} \ \bigg(\mathcal{L}_X\Omega^{\alpha_-}=(r_u^-+r_s^-)\Omega^{\alpha_-}\bigg).$$
In other words,
$$div_X\Omega^{\alpha_+}=r_u^++r_s^+ \ \bigg(div_X\Omega^{\alpha_-}=r_u^-+r_s^-\bigg).$$
\end{lemma}

In the following, the flow being $C^1$ suffices. However, the proof in that generality would require subtle approximation techniques of $\cite{hoz3}$, since we cannot assume $C^1$-regularity of the weak stable and unstable bundles, when the derivative of the flow is not Hölder continuous. For the sake of simplicity, we assume the flow to be $C^{1+}$.

\begin{theorem}\label{contchar}
Let $X$ be the generating vector field for a projectively Anosov flow. Then, the followings are equivalent:

(1) $X$ is Anosov.

(2) There exists a positive contact form $\alpha_+$, such that for some $\xi_-$, the pair $(\xi_-,\xi_+:=\ker{\alpha_+})$ is a supporting bi-contact structure and $-\alpha_+ \wedge d\alpha_+<(div_X\Omega^{\alpha_+})\Omega^{\alpha_+} < \alpha_+\wedge d \alpha_+.$

(3) There exists a negative contact form $\alpha_-$, such that for some $\xi_+$, the pair $(\xi_-:=\ker{\alpha_-},\xi_+)$ is a supporting bi-contact structure and $\alpha_- \wedge d\alpha_-<(div_X\Omega^{\alpha_-})\Omega^{\alpha_-} < -\alpha_-\wedge d \alpha_-.$
\end{theorem}

\begin{proof}
We prove the equivalence of (1) and (2). Showing the equivalence of (1) and (3) is similar.

Assume (2) and let $r^+_u$ and $r^+_s$ be the associated expansion rates, for some projectively Anosov flow supported by $(\xi_-,\xi_+)$, induced by $\alpha_+$ as in Remark~\ref{inducedvol}. Using Lemma~\ref{contcomp} and Lemma~\ref{volcomp}, we can translate the condition on $\alpha_+$ to
$$r^+_s-r^+_u<r^+_s+r^+_u<r^+_u-r^+_s.$$
This yields $r^+_s<0$ and $r^+_u>0$, implying the Anosovity of $X$.

Now, we prove the other implication, utilizing a similar idea as above. Without loss of generality, we assume the norm satisfying the Anosovity condition $r_s<0<r_u$ to be $C^1$.

Define the 1-forms $\alpha_u$ and $\alpha_s$ by letting $\alpha_u\big|_{E^s}=\alpha_s\big|_{E^u}=0$ and $\alpha_u(e_u)=\alpha_s(e_s)=1$, where $e_s\in E^s\cap \eta\subset TM$ and $e_u\in E^u \cap \eta \subset TM$ are unit vector fields (see Remark~\ref{shifting} and \ref{formshifting} and notice that any choice of $\eta$ induces an appropriate metric on $TM$, with respect to which, the expansion rates are as desired) and $(e_s,e_u,X)$ is an oriented basis. Therefore, the expansion rates induced by the $C^1$ positive contact form $\alpha_+:=\alpha_u-\alpha_s$ (as in Remark~\ref{inducedvol}) are the same as $r_s$ and $r_u$, and therefore satisfying $r_s<0<r_u$, or equivalently, $r_s-r_u<r_s+r_u<r_u-r_s$. Lemma~\ref{contcomp} and~\ref{volcomp} yield (2) (It is noteworthy that the above construction of $\alpha_+$ does not depend on the choice of the oriented basis $(e_s,e_u,X)$, i.e. if we use an oriented basis of the form $(e_u,e_s,X)$, we would need to define $\alpha_+=\alpha_u+\alpha_s$).

\end{proof}

%%%%%%%%%%%%%%%
\section{Invariant volume forms, $(-1)$-Cartan structures and \\ Liouville reparametrization}\label{5}

In this section, we study the symmetries that the existence of an invariant volume form implies on the geometry of a volume preserving Anosov flow. This gives us various characterizations of an Anosov flow being volume preserving, in terms of the theory of contact hyperbolas, the Reeb dynamics of the supporting contact structures, and Liouville geometry.

In what follows, by a {\em volume preserving} Anosov flow, we mean one which preserves a continuous volume form. However, it is known that for a $C^k$ flow, the $C^k$ version of Livshitz regularity theorem implies any such continuous volume form to be $C^k$ \cite{hflows,livreg,llave}.

We first note that $E^s$ and $E^u$ are $C^1$ plane fields, when $X$ is a $C^{1+}$ Anosov flow (see \cite{hflows,hruder,regular}). This is an important fact in what follows, since it helps us preserve the metric symmetries of a volume preserving Anosov flow, when translating to the framework of contact geometry (for Anosov flows of lower regularity, one would require the approximation techniques of \cite{hoz3}, which do not automatically respect such symmetry).

Let $e_u\in E^u\subset TM$ be a unit vector field with respect to a metric, satisfying the Anosovity condition and let $\alpha_u$ be a 1-form, such that $\alpha_u\big|_{E^s}=0$ and $\alpha_u(e_u)=1$. Possibly after a perturbation, we can assume $\alpha_u$ to be $C^1$ and we have $r_u:=\alpha_u([e_u,X])>0$. That is, the induced expansion rate of the unstable direction is positive (see Remark~\ref{shifting} and \ref{formshifting}).

Let $\Omega$ be a continuous (and therefore, $C^{1}$) positive volume form which is invariant under the flow, and define $\alpha_s:=\Omega(.,e_u,X)$. Note that $\alpha_s$ is a $C^1$ 1-form, whose kernel is $E^s$. Since $div_X\Omega=0$, by Theorem~\ref{metric1} we have $r_s:=\alpha_s([e_s,X])=-r_u<0$.

Now, define $\alpha_+:=\alpha_u-\alpha_s$ and notice that $\alpha_+$ is a positive contact form, since its induced expansion rates satisfy $r_u-r_s=2r_u=-2r_s>0$. Similarly, define the negative contact structure $\alpha_-=\alpha_u+\alpha_s$. Therefore, for any volume preserving Anosov flow, we have a supporting bi-contact structure $(\xi_-:=\ker{\alpha_-}=\ker{\alpha_u+\alpha_s},\xi_+:=\ker{\alpha_+}=\ker{\alpha_u-\alpha_s})$, which captures the symmetry of an invariant volume form.

Furthermore, by solving $d\alpha_+(R_{\alpha_+},X)=d\alpha_-(R_{\alpha_-},X)=0$ and , one can easily show that
$$R_{\alpha_+}\subset \langle e_u-e_s,X \rangle=\xi_- \text{ \ \ and \ \ } R_{\alpha_-}\subset \langle e_u+e_s,X \rangle=\xi_-.$$

%%%%%%%

%%%%%%%%%%%%%%%%%%%%
\subsection{Taut contact hyperbolas and a Reeb dynamical characterization}\label{5.2}

As it can be seen in the discussion above, additional geometric symmetry can be observed in the case of volume preserving Anosov flows. In this section, we describe this extra structure, in terms of the theory of {\em contact hyperbolas}, developed by Perrone~\cite{perrone}, following the similar theory of {\em contact circles} by Geiges-Gonzalo \cite{circle2,circle}.

A {\em contact hyperbola} on $M$ is a pair of positive and negative contact forms $(\alpha_1,\alpha_2)$, such that $\alpha_a:=a_1\alpha_1+a_2\alpha_2$ is also a contact form, for any $a:=(a_1,a_2)\in \mathbb{H}_r^1$, where $H_r^1=\{(a_1,a_2)|a_1^2-a_2^2=r\}$ for $r\in \{-1,1\}$. Furthermore, a contact hyperbola is called {\em taut}, if $\alpha_a\wedge d\alpha_a=r\alpha_1\wedge d\alpha_1$ (or equivalently, $\alpha_a\wedge d\alpha_a=-r\alpha_2\wedge d\alpha_2$), for any $a\in \mathbb{H}_r^1$. It is easy to show \cite{perrone} that $(\alpha_1,\alpha_2)$ is a taut contact hyperbola, if and only if,
$$\alpha_1\wedge d \alpha_1=-\alpha_2\wedge d\alpha_2\ \ \ \text{and}\ \ \ \alpha_1\wedge d \alpha_2=-\alpha_2\wedge d\alpha_1.$$

Notice that if $(\alpha_1,\alpha_2)$ is a taut contact hyperbola, then $\ker{\alpha_1}$ and $\ker{\alpha_2}$ form a bi-contact structure if transverse, with any flow directing the intersection of them being projectively Anosov. It is known that the converse is not true, i.e. there are projectively Anosov flows which do not come from a contact hyperbola.

As it is seen above, for a volume preserving Anosov flow, the supporting bi-contact structure $(\ker{ \alpha_-}=\ker{\{\alpha_u+\alpha_s\}},\ker{\alpha_+}:=\ker{\{ \alpha_u-\alpha_s\} })$ exists, where $\alpha_u\big|_{E^s}=\alpha_s\big|_{E^u}=0$, and the induced positive volume forms and the expansion rates satisfy $\Omega^{\alpha_-}=\Omega^{\alpha_+}=\alpha_s\wedge\alpha_u\wedge\alpha_X$ (for any 1-form $\alpha_X$ with $\alpha_X(X)=1$) and $r_s^+=r_s^-=-r_u^+=-r_u^-$, respectively (see Remark~\ref{shifting} and \ref{formshifting}). By Lemma~\ref{contcomp}, we have
$$\alpha_+\wedge d\alpha_+=(r^+_u-r^+_s)\Omega^{\alpha_+}=(r^-_u-r^-_s)\Omega^{\alpha_-}=-\alpha_-\wedge d \alpha_-.$$

Moreover, the discussion in the beginning remarks of this section shows that $R_{\alpha_+}\subset\xi_-$ and $R_{\alpha_-}\subset \xi_+$, yielding
$$\alpha_+\wedge d\alpha_-=-\alpha_-\wedge d\alpha_+=0.$$

Therefore, $(\alpha_-,\alpha_+)$ is a taut contact hyperbola in this case. In fact, it satisfies the stronger condition of $\alpha_+\wedge d\alpha_-=-\alpha_-\wedge d\alpha_+=0$. A taut contact hyperbola with this property is called a {\em $(-1)$-Cartan structure \cite{perrone}}. It turns out that, not only this can be improved to a geometric characterization of volume preserving flows, it will also give us a characterization, purely in terms of the underlying Reeb flows.

\begin{theorem}\label{cartan}
Let $\phi^t$ be a projectively Anosov flow on $M$. Then, the followings are equivalent:

(1) The flow $\phi^t$ is a volume preserving Anosov flow.

(2) There exists a supporting bi-contact structure $(\xi_-,\xi_+)$ and contact forms $\alpha_-$ and $\alpha_+$ for $\xi_-$ and $\xi_+$, respectively, such that $(\alpha_-,\alpha_+)$ is a $(-1)$-Cartan structure.

(3) There exists a supporting bi-contact structure $(\xi_-,\xi_+)$ and Reeb vector fields $R_{\alpha_-}$ and $R_{\alpha_+}$ for $\xi_-$ and $\xi_+$, respectively, such that $R_{\alpha_-}\subset \xi_+$ and $R_{\alpha_+}\subset \xi_-$.
\end{theorem}

\begin{proof}
The above discussion shows that (1) implies (2). We can also conclude (3) from (2), by noticing that $\alpha_-\wedge d\alpha_+=-\alpha_+\wedge d\alpha_-=0$ yields $R_{\alpha_-}\subset \xi_+$ and $R_{\alpha_+}\subset \xi_-$. Therefore, the only remaining part is to show that a projectively Anosov flow is volume preserving, with the assumptions of (3).

We first note that any projectively Anosov flow with $R_{\alpha_-}\subset \xi_+$ and $R_{\alpha_+}\subset \xi_-$ is Anosov, thanks to Theorem~6.3 of \cite{hoz3}. Let $\alpha_-$ and $\alpha_+$ be the contact forms in (3). As in Remark~\ref{inducedvol}, we consider the expansion rates $r_s$ and $r_u$, induced from the decomposition $\alpha_+=\alpha_u-\alpha_s$. Notice that we can write $\alpha_-=f\alpha_u+g\alpha_s$, for positive $X$-differentiable functions $f,g>0$. One can easily check by solving $d\alpha_+(R_{\alpha_+},X)=0$ and using the fact that $\alpha_+(R_{\alpha_+})=1$ (or as it is done in Section~6 of \cite{hoz3}) that we can write ($q_+$ being a real function)
\begin{equation}\label{eq1}
R_{\alpha_+}=\frac{1}{r_u-r_s} (-r_se_u-r_ue_s)+q_+ X.
\end{equation}
and from $\alpha_-(R_{\alpha_+})=0$, we get 
\begin{equation}\label{eq2}
g=-\frac{fr_s}{r_u}.
\end{equation}

Also, using $\alpha_+(R_{\alpha-})=0$ and $\alpha_-(R_{\alpha_-})=1$, we can write ($q_-$ being a real function)
\begin{equation}\label{eq3}
R_{\alpha_-}=\frac{1}{f+g}(e_u+e_s)+q_- X.
\end{equation}

On the other hand, we have 
$$0=d\alpha_-(R_{\alpha_-},X)=d\alpha_-(e_s+e_u,X)=-X\cdot\alpha_-(e_s,e_u)-\alpha_-([e_s+e_u,X])$$

\begin{equation}\label{eq5}
\Rightarrow X\cdot f+X\cdot g+gr_s+fr_u=0
\end{equation}

Using Equation~\ref{eq2} in the above, we get
$$X\cdot f+X\cdot g+(f+g)(r_s+r_u)=0,$$
which yields
\begin{equation}\label{eq6} r_s+r_u=-X\cdot (f+g).\end{equation}

Finally, in order to show that the flow is volume preserving, it suffices to define $\Omega:=e^{f+g}\Omega^{\alpha_+}$ and use Lemma~\ref{contcomp} Equation~\ref{eq6} and to compute
$$div_X\Omega=X\cdot (f+g)e^{f+g}\Omega^{\alpha_+} + e^{f+g} (div_X\Omega^{\alpha_+})\Omega^{\alpha_+}$$
$$=e^{f+g}\big(X\cdot (f+g)+r_s+r_u\big)\Omega^{\alpha_+}=0$$

\end{proof}

\begin{remark}\label{reebreg}
It is noteworthy that the above proof shows that in fact, for a volume preserving Anosov flow, we have a function worth of pairs of Reeb vector fields satisfying (3) of Theorem~\ref{cartan}. This is useful in particular, when we require higher regularity of the underlying contact geometry. More precisely, the contact forms in Theorem~\ref{cartan} (2) are, except in the case of algebraic Anosov flows, only $C^1$. Therefore, their Reeb vector fields can be only assumed to be $C^0$ in general (this is due to the result of Ghys \cite{ghys}, which asserts that except in the case of algebraic Anosov flows, the weak stable and unstable bundles cannot be $C^2$). However, if we let go of the {\em symmetry} of the contact forms in Theorem~\ref{cartan} (2), we can achieve higher regularity of the underlying contact geometry in the following sense.

Let $\phi^t$ be a volume preserving Anosov flow and $\alpha_+$ a $C^\infty$ contact form, which $C^1$-approximates $\alpha_u-\alpha_s$ and $\alpha_+(X)=0$ ($\alpha_u$ and $\alpha_s$ are as in the above theorem). Note that $R_{\alpha_+}$ is $C^\infty$ and $\xi_-:=\langle R_{\alpha_+},X \rangle$ is a $C^\infty$ negative contact structure, if the $C^1$-approximation of $\alpha_u-\alpha_s$ is small enough. Now, a negative contact form for $\xi_-$ is of the form $\alpha_-:=f\alpha_u+g\alpha_s$, and in order to have $R_{\alpha_-}\subset \xi_+:=\ker{\alpha_+}$, it suffices to choose any $f$ and $g$ satisfying Equation~\ref{eq2}, where $r_s$ and $r_u$ are the expansion rates associated with $\alpha_+$ (as in Remark~\ref{inducedvol}). We note that although $\xi_-$ is $C^\infty$, the contact form $\alpha_-$ can only be assumed to be $C^1$ in general.
\end{remark}

\begin{corollaryproof}
In Theorem~\ref{cartan} (3), the supporting bi-contact structure $(\xi_-,\xi_+)$, as well as at least one of $R_{\alpha_+}$ or $R_{\alpha_-}$ can be chosen to be $C^\infty$.
\end{corollaryproof}

Theorem~\ref{cartan} gives examples of $(-1)$-Cartan structures and taut contact hyperbolas, whenever we have volume preserving flows, including the case of algebraic Anosov flows and more interestingly, we get new examples on hyperbolic manifolds, using the examples of contact Anosov flows on those manifolds~\cite{foulon}. The construction also yields new examples of $(-1)$-Cartan structures on any manifold supporting a transitive Anosov flow, many of which are not contact \cite{bbu,asaokaverj}.

\begin{corollary}
There exist infinitely many hyperbolic manifolds which admit a $(-1)$-Cartan structure (and in particular, a taut contact hyperbola).
\end{corollary}

%%%%%%%%%%%%%%%%%%%%%
\subsection{From the viewpoint of Liouville geometry}\label{5.3}

In \cite{hoz3}, we have shown how from an Anosov flow $X$ on a 3-manifold $M$, we can construct two {\em Liouville pairs}, $(\alpha_-,\alpha_+)$ and $(-\alpha_-,\alpha_+)$, where $(\xi_-:=\ker{\alpha_-},\xi_+:=\ker{\alpha_+})$ is a supporting bi-contact structure for $X$. That is, $\omega_1:=d\alpha_1$ and $\omega_2:=d\alpha_2$ are exact symplectic structures on $[-1,1]_t\times M$, where $\alpha_1:=(1-t)\alpha_-+(1+t)\alpha_+$ and $\alpha_2:=(1-t)\alpha_--(1+t)\alpha_+$.

Recall that $(W,d\alpha)$ is an {\em exact symplectic 4-manifold}, if $W$ is an oriented 4-manifold (with boundary) and $d\alpha$ is an exact symplectic structure on $W$, i.e. $d\alpha\wedge d\alpha>0$. For any exact symplectic manifold $(W,d\alpha)$, there exists a unique vector field $Y$, such that $\iota_Y d\alpha=\alpha$, or equivalently $\mathcal{L}_Y d\alpha=d\alpha$. Such vector field is call a {\em Liouville vector field}, if it points in the outward direction on $\partial W$, and the pair $(W,Y)$ is called a {\em Liouville structure}. We note that this is the case for an exact symplectic manifold constructed from an Anosov 3-flows above, since it is a {\em symplectic filling} for the contact manifold $(M,\xi_+)\cup (-M,\xi_-)$.

The relation between the associated Liouville vector field and the underlying Anosov vector field is more subtle in the general case of $C^1$ Anosov flows. But for $C^{1+}$ volume preserving Anosov flows, such connection becomes very straightforward, thanks to the symmetries implied by the existence of an invariant volume form and the fact that in this case, the weak stable and unstable bundles are $C^1$ \cite{hruder,regular}.

In what follows, we let $(\alpha_-=\alpha_u+\alpha_s,\alpha_+=\alpha_u-\alpha_s)$ be the $(-1)$-Cartan structure of Theorem~\ref{cartan} (in particular, $\Omega^{\alpha_+}=\alpha_s\wedge\alpha_u\wedge \alpha_X$ is a positive volume form for any 1-form $\alpha_X$ with $\alpha_X(X)=1$). We have the 1-form $\alpha_1=(1-t)\alpha_-+(1+t)\alpha_+=2\alpha_u-2t\alpha_s$ on $ [-1,1]_t\times M$, and compute
$$d\alpha_1\wedge d\alpha_1=\big\{ 2d\alpha_u-2dt\wedge \alpha_s-2td\alpha_s \big\} \wedge \big\{ 2d\alpha_u-2dt\wedge \alpha_s-2td\alpha_s \big\}$$
$$=-4dt\wedge \alpha_s\wedge d\alpha_u=4r_udt\wedge \Omega^{\alpha_+},$$
implying that $( [-1,1]_t\times M,d\alpha_1)$ is an exact symplectic manifold.  Similarly, we can show $d\alpha_2\wedge d\alpha_2=-4r_sdt\wedge \Omega^{\alpha_-}$, yielding another exact symplectic structure on $[-1,1]_t\times M$. Notice that  with the above assumptions, we have $r_u=-r_s>0$ and $\Omega^{\alpha_+}=\Omega^{\alpha_-}$. But the fact that the weak stable and unstable bundles are $C^1$ in this case, thanks to the Hölder continuity of the derivatives of the flow, plays a crucial role in preserving the symmetry, when going from a metric description of the underlying Anosov flow to a contact geometric one, and hence, significantly simplifying the construction of the above Liouville pairs, compared to Anosov flows of lower regularity studied in~\cite{hoz3}. We also remark that, unless $X$ is an algebraic Anosov flow \cite{ghys}, the $(-1)$-Cartan structure is a priori only $C^1$, and therefore, the 2-forms $d\alpha_1$ and $d\alpha_2$ above are exact symplectic structures, only in the $C^0$ sense. 

Now, if we define the vector field $Y_1:=\frac{1}{r_u}X+2t\partial_t$, we can compute

$$\iota_{Y_1} d\alpha_1=\frac{2}{r_u} \iota_X d\alpha_u -4t\alpha_s-\frac{2t}{r_u} \iota_Xd\alpha_s=2\alpha_u-4t\alpha_s+2t\alpha_s=\alpha_1.$$

Therefore, $Y_1$ is the Liouville vector field for $([-1,1]_t\times M,d\alpha_1)$. Notice that a similar computation helps us compute the Liouville vector field of $([-1,1]_t\times M,d\alpha_2)$. It can be seen \cite{simic} that the 1-forms $\alpha_u$ and $\alpha_s$ can be chosen such that $r_u$ and $r_s$ are $C^1$. In that case, it is noteworthy and surprising that although the constructed symplectic structures above are a priori only $C^0$, their corresponding Liouville vector fields are $C^1$.

Now, we can consider the vector field $X_L:=\frac{1}{r_u}X$, which generates a reparametrization of the original flow. Its associated expansion rates are $r'_u=\frac{r_u}{r_u}=1$ and $r'_s=\frac{r_s}{r_u}=-1$, respectively (see Remark~3.18 of \cite{hoz3}) and we have $Y_1\big|_{\{ 0\}\times M}=X_L$. We call such $X_L$ {\em the Liouville reparametrization of a volume preserving Anosov 3-flow} (in the sense of \cite{simic}, this is the {\em synchronization} of the flow with respect to both stable and unstable directions, simultaneously). 

The following theorem proves that the Liouville reparametrization of a volume preserving Anosov flow has even a closer relation to the underlying Reeb dynamics of Theorem~\ref{cartan}

\begin{theorem}\label{liou}
Let $X$ be the generating vector field of a  volume preserving Anosov flow. If $X_L$ is the generating vector field for the Liouville reparametrization of the flow, the following holds:

(1) The flow generated by $X_L$ preserves the transverse plane field $\langle R_{\alpha_-},R_{\alpha_+} \rangle$, where $R_{\alpha_-}$ and $R_{\alpha_+}$ are the Reeb vector fields of Theorem~\ref{cartanintro} (2);

(2) The pair $(M,X_L)$ can be extended to a Liouville structure $([-1,1]\times M,Y)$, such that $([-1,1]\times M,Y)\big|_{\{ 0\} \times M}=(M,X_L).$
\end{theorem}

\begin{proof}
The above argument yields (2). In order to prove (1), let ($\alpha_-=\alpha_u+\alpha_s,\alpha_+=\alpha_u-\alpha_s$) be the $(-1)$-Cartan structure of part (2) in Theorem~\ref{cartan}. We have
$$\mathcal{L}_{X_L}\alpha_+=\mathcal{L}_{X_L}(\alpha_u-\alpha_s)=\alpha_u+\alpha_s=\alpha_-.$$

Similarly, one can show $\mathcal{L}_{X_L}\alpha_-=\alpha_+$. Define the 1-form $\alpha_{X_L}$ by letting $\alpha_{X_L}(X_L)=1$ and $\alpha_{X_L}(\langle R_{\alpha_-},R_{\alpha_+}\rangle)=0$. The goal is prove $\mathcal{L}_{X_L}\alpha_{X_L}=0$. Note that by construction, $\alpha_{X_L}$ is differentiable along the flow and $(\mathcal{L}_{X_L}\alpha_{X_L})\wedge \alpha_{X_L}=0$.

Also, by plugging the basis $(R_{\alpha_-},R_{\alpha_+},X_L)$, we can observe
$$d\alpha_+=\alpha_{X_L} \wedge \alpha_- \ \ \ \text{ and }\ \ \ d\alpha_-=\alpha_{X_L}\wedge \alpha_+,$$
which implies
$$(\mathcal{L}_{X_L}\alpha_{X_L})\wedge \alpha_-=\mathcal{L}_{X_L}(\alpha_{X_L}\wedge \alpha_-)-\alpha_{X_L}\wedge \mathcal{L}_{X_L}\alpha_-$$
$$=\mathcal{L}_{X_L} d\alpha_+-\alpha_{X_L}\wedge\alpha_+=d(\mathcal{L}_{X_L} \alpha_+)-\alpha_{X_L}\wedge\alpha_+=d\alpha_--d\alpha_-=0.$$

Similarly, we have $(\mathcal{L}_{X_L}\alpha_{X_L})\wedge \alpha_+=0$. This yields $\mathcal{L}_{X_L}\alpha_{X_L}=0$, completing the proof.

\end{proof}

%%%%%%%%%%%%%%%%%%

\section{Applications to bi-contact surgeries}\label{6}

In this section, we discuss the implications of our work in the surgery theory of Anosov flows. Theorem~\ref{cartan} shows that for volume preserving Anosov flows, the Reeb vector fields associated with the supporting bi-contact structure $(\xi_-=\ker{\alpha_-},\xi_+=\ker{\alpha_+})$ can be contained in one another. In this case, if we push a periodic orbit of the flow $\gamma_0$, which is a Legendrian knot for both $\xi_-$ and $\xi_+$, along one of these Reeb vector fields, say $R_{\alpha_+}$, it stays Legendrian for $\xi_+$ (since $R_{\alpha_+}$ preserves $\xi_+$) and it immediately becomes transverse to $\xi_-$ (since $R_{\alpha_+}$ is a Legendrian vector field for $\xi_-$). We call such a knot a {\em Legendrian-transverse} knot.

In \cite{salmoi,salmoi2}, Salmoiraghi develops two flavors of {\em bi-contact surgery} operations in a neighborhood of a Legendrian-transverse knot. One of these operations is applied by cutting the manifold on an annulus, which is tangent to the flow and contains the Legendrian-transverse knot, and then glueing back using a Dehn twist \cite{salmoi}. The other operation is applied to a transverse annulus containing such knot \cite{salmoi2}. Moreover, he shows that using the coordinates coming from the above argument on the Reeb vector field, one can reconstruct the classical {\em Goodman surgery} in the neighborhood of a periodic orbit of an Anosov flow, using the bi-contact surgery of \cite{salmoi2}. However, notice that in the above argument, it suffices for the Reeb vector field of just one of the contact structures to be contained in the other contact structure, only in a small neighborhood of the periodic orbit one wants to apply the Goodman surgery on. In the following, we show that this is possible for {\em any} (possibly non volume preserving) Anosov flow. The main idea is to show that one can assume that the flow has constant divergence along a fixed periodic orbit.

\begin{theorem}\label{surg}
Let $\phi^t$ be an Anosov flow. Given any periodic orbit $\gamma_0$, there exists a supporting bi-contact structure $(\xi_-,\xi_+=\ker{\alpha_+})$, such that we have $R_{\alpha_+}\subset \xi_-$ in a regular neighborhood of $\gamma_0$. Therefore, there exists an isotopy $\{\gamma_t\}_{t\in [0,1]}$, which is supported in an arbitrary small neighborhood of $\gamma_0$, and $\gamma_t$ is a  Legendrian-transverse knot for any $0<t\leq 1$.
\end{theorem}

\begin{proof}
As discussed above, it is enough to show that there exists a tubular neighborhood $N(\gamma_0)$ and a pair of contact forms $\alpha_+$ and $\alpha_-$, such that $(\ker{\alpha_-},\ker{\alpha_+})$ is a supporting bi-contact structure for $X$ and $\alpha_-(R_{\alpha_+})=0$. It is easy to show that this is would have been possible, if the associated expansion rates were constant. The idea of the proof is to find an appropriate norm, which satisfies this condition on $\gamma_0$ and use the openness of the contact condition. To do so, we need an approximation technique similar to one used in the main theorem of \cite{hoz3}. The only caveat is that we need our approximation not to affect the preassigned norm on $\gamma_0$.

Let $T$ be the period of $\gamma_0$ and $\lambda^{\gamma_0}_u$ and $\lambda^{\gamma_0}_s$ be the eigenvalues of the return map along $\gamma_0$, corresponding to the unstable and stable directions, respectively. We can choose a $X$-differentiable norm on $TM/\langle X \rangle\big|_{\gamma_0}$, such that the induced expansion rates $r_u\big|_{\gamma_0}$ and $r_s\big|_{\gamma_0}$ are constants satisfying $e^{r_uT}=\lambda^{\gamma_0}_u$ and $e^{r_sT}=\lambda^{\gamma_0}_s$. We can then extend such norm to some $X$-differentiable norm on $TM/\langle X\rangle\simeq E^s\oplus E^u$ in a neighborhood of $\gamma_0$. Let $N(\gamma_0)$ be a possibly smaller neighborhood, on which $r_s<0<r_u$.

We define the %$C^0$ 
1-forms $\tilde{\alpha}_u$ and $\tilde{\alpha}_u$, by letting $\tilde{\alpha}_u(E^s)=\tilde{\alpha}_s(E^u)=0$ and $\tilde{\alpha}_u(e_u)=\tilde{\alpha}_s(e_s)=1$, where $e_s \in E^s$ and $e_u \in E^u$ are the unit vectors with respect to our norm. We can $C^0$-approximate $\tilde{\alpha}_u$ and $\tilde{\alpha}_u$ by $C^\infty$ 1-forms $\bar{\alpha}_u$ and $\bar{\alpha}_u$ and find $X$-differentiable functions $f_u$ and $f_s$ such that $f_u\bar{\alpha}_u(e_u)=f_s\bar{\alpha}_s(e_s)=1$. Using the following lemma, we can approximate these functions with appropriate $C^1$ functions to serve our goal.

\begin{lemma}\label{apporbit}
If $f$ is $X$-differentiable and $\eta$-Hölder continuous and $\gamma$ is a periodic orbit of $X$ (a $C^1$ flow on $n$-dimensional closed manifold $M$). Then, for any $\epsilon>0$, there exists a $C^1$ function $\bar{f}$, such that $f\big|_{\gamma}=\bar{f}\big|_{\gamma}$ and we have $|f-\bar{f}|<\epsilon$ and $|X\cdot f-X\cdot\bar{f}|<\epsilon$.
\end{lemma}

\begin{proof}
Let $N_\delta(\gamma)$ be a sufficiently small tubular neighborhood of $\gamma$, on which the function $d(x)$, measuring the distance of $x\in M$ from $\gamma$, is $C^1$, i.e. $N_\delta=\{x\in M | d(x)<\delta \}$. Let $\bar{d}(x)$ be any $C^1$ function on $M$, where $\bar{d}(x)=d(x)$ on $N_{\frac{\delta}{2}}(\gamma) \subset N_\delta(\gamma)$ and $\bar{d}(x)\neq 0$ everywhere.

 Now, we can write $f(x)=f^\gamma(x)+\bar{d}^{\frac{\eta}{2}}(x)g(x)$, where $f^\gamma(x)$ is any $C^1$ extension of $f\big|_{\gamma}$ on $M$ and $g(x)$ is well-defined, continuous and $X$-differentiable function on $M\backslash \gamma$. We extend $g$ to $M$ by letting $g(\gamma)=0$.

\begin{claim}
The function $g$ is continuous and $X$-differentiable on $N_\delta(\gamma)$.
\end{claim}

\begin{proof}
$$\lim_{d(x)\rightarrow 0}g(x)=\lim_{d(x)\rightarrow 0}\frac{f(x)-f^\gamma(x)}{d^{\frac{\eta}{2}}(x)}=\lim_{d(x)\rightarrow 0}\frac{f(x)-f^\gamma(x)}{d^\eta(x)} d^{\frac{\eta}{2}}(x)=0.$$

The last equality follows from $f$ being $\eta$-Hölder continuous. Therefore, $g$ is a continuous function on $N_\delta(\gamma)$ (in fact it is $\frac{\eta}{2}$-Hölder continuous). Moreover, g is $X$-differentiable in this neighborhood, since we have $X\cdot g\big|_\gamma=0$.
\end{proof}

Now, we use Lemma~4.2 of \cite{hoz3} to find a $C^1$ function $\bar{g}$, where $|g-\bar{g}|$ and $|X\cdot g-X\cdot\bar{g}|$ are arbitrary small. In fact, if we define $\bar{f}:=f^{\gamma_0}+\bar{d}^{\frac{\eta}{2}}\bar{g}$, we can find an approximation of $g$, such that $\bar{f}$ is the desired $C^1$ function. This completes the proof of Lemma~\ref{apporbit}.

\end{proof}

Let $\bar{f}_s$ and $\bar{f}_u$ be the approximations of $f_s$ and $f_u$ as in Lemma~\ref{apporbit}. As in \cite{hoz3}, we can define the $C^1$ contact forms $\alpha_+$ with $e'_s\in E^s\subset TM$, $e'_u \in E^u \subset TM$, $r'_u$ and $r'_s$ induced by $\alpha_+$ as in Remark~\ref{inducedvol}, such that when restricted to $\gamma_0$, we have $\alpha_+(e_u)=\bar{f}_u\bar{\alpha}_u(e_u)=f_u\bar{\alpha}_u(e_u)=1$ and similarly $\alpha_+(e_s)=1$, which yields $e_s=e'_s$, $e_u=e'_u$, $r_s=r'_s$ and $r_u=r'_u$ (we refer the reader to \cite{hoz3} for the the technical details of the approximations used in the definition. It is enough for us to know that the induced unit vectors and expansion rates from these approximating contact forms are arbitrary close to the ones we started with, while agreeing on $\gamma_0$ ). As in Equation~\ref{eq1}, we have $R_{\alpha'_+}=\frac{1}{r'_u-r'_s}\{-r'_ue'_s-r'_se'_u\}+q_+ X$, for some real function $q_+$. Let $\xi'_-:=\langle R_{\alpha_+},X\rangle$.

\begin{claim}
There exists a regular neighborhood $N(\gamma_0)$, on which $\xi'_-$ is a negative contact structure.
\end{claim}

\begin{proof}
Choose a 1-form $\alpha_-$ such that $\xi'_-:=\ker{\alpha_-}$ and $\alpha_-(e'_s+e'_u)>0$. Compute

$$d\alpha_-(R_{\alpha_+},X)=\frac{1}{r'_u-r'_s}\alpha_-([X,-r'_ue'_s-r'_se'_u])$$
$$\Rightarrow d\alpha_-(R_{\alpha_+},X)\big|_{\gamma_0}=\frac{r'_sr'_u}{r'_u-r'_s}\alpha_-(e'_s+e'_u)<0,$$
where in the last equality, we have used the fact that by construction, we have $X\cdot r'_s=X\cdot r_s=X\cdot r'_u=X\cdot r_u=0$ on $\gamma_0$. Thanks to the openness of the contact condition, $\xi'_-$ is a negative contact structure in some tubular neighborhood $N(\gamma_0)$.
\end{proof}

We can extend $\xi'\big|_{N(\gamma_0)}$ to some negative contact structure $\xi_-$ on $M$, such that the supporting bi-contact structure $(\xi_-,\xi_+=\ker{\alpha_+})$ has the desired properties.
\end{proof}

%\begin{remark}
%Goodman \cite{goodman} and Fried \cite{fried} surgeries on the periodic orbits of Anosov flows were both introduced in the early 1980s. Since then, the two operations were assumed by many authors to produce orbit equivalent Anosov flows and hence, the use of the term {\em Goodman-Fried surgery} has been common. However, establishing such equivalence is proven to be more subtle than assumed. In fact, this was proven only recently in the case of transitive Anosov flow by Shannon \cite{shannon}. In \cite{salmoi2}, Salmoiraghi claims that in an upcoming paper, he uses the Theorem~\ref{surg} above to establish the equivalence in the general case.
%\end{remark}

\begin{corollary}
The bi-contact surgeries of Salmoiraghi \cite{salmoi,salmoi2} can be applied in an arbitrary small neighborhood of a periodic orbit of any  Anosov flow. In particular, the bi-contact surgery of \cite{salmoi2} reconstructs the Goodman surgery.
\end{corollary}

%%%%%%%%%%%%%%%%%%%%%%%%%%%%%

\Addresses
\end{document}